\documentclass[12pt]{article}
\usepackage{hyperref}
\usepackage{amsthm}
\usepackage{enumitem}
\usepackage{amssymb,amsmath,makeidx,verbatim}
\usepackage{makeidx}	
\makeindex
\usepackage{graphicx}
\newcommand{\be}{\begin{enumerate}}
\newcommand{\ee}{\end{enumerate}}
\newcommand{\bq}{\begin{eqnarray*}}
\newcommand{\eq}{\end{eqnarray*}}

\numberwithin{equation}{section}
\newtheorem{theorem}{Theorem}[section]
\newtheorem{corollary}[theorem]{Corollary}

\newtheorem{lemma}[theorem]{Lemma}

\newtheorem{proposition}[theorem]{Proposition}
\begin{document}
\date{}
\baselineskip 14pt
\newcommand{\disp}{\displaystyle}
\thispagestyle{empty}
\title{Differential Harnack estimates for conjugate heat equation under  the  Ricci flow}
\author{Abimbola Abolarinwa\thanks{Department of Mathematics,  University of Sussex, Brighton, BN1 9QH, United Kingdom.}  \thanks{E-mail: a.abolarinwa@sussex.ac.uk}}
\maketitle

\begin{abstract}
We prove certain localized and global differential Harnack inequality for all positive solutions to the geometric conjugate heat equation coupled to the forward  in time Ricci flow. In this case,  the diffusion operator is perturbed with the curvature operator,  precisely, the Laplace-Beltrami operator is replaced with "$ \Delta - R(x,t)$", where $R$ is the scalar curvature of the Ricci flow, which is  well generalised  to the case of nonlinear heat equation with potential. Our estimates improve on some well known results by weakening the curvature constraints. As a by product,  we obtain some Li-Yau type differential  Harnack estimate. The localized version of our estimate is very useful in extending the results obtained to noncampact case.
\end{abstract} 

\section{Introduction} 
 Let $M$ be an $n$-dimensional compact (or noncompact without boundary) manifold on which a one parameter family of Riemannian metrics $g(t), t \in [0, T)$ is defined. We say $(M, g(t))$ is a solution to the Ricci flow if it is evolving by the following nonlinear weakly parabolic partial differential equation
\begin{equation}\label{eq11}
\frac{\partial }{ \partial t} g (x, t)  = - 2 Ric(x,t),  \hspace{2cm} (x, t) \in M \times [0, T)
\end{equation}
with $g(x, 0) = g(0)$, where $Ric$ is the Ricci curvature and $T \leq \infty$. By the positive solution to the heat equation on the manifold, we mean a smooth function atleast $C^2$ in $x$ and $C^1$ in $t$,  $ u \in C^{2, 1}(M \times [0,T])$  which satisfies the following equation 
\begin{equation}\label{eq12}
\Big( \Delta - \frac{\partial }{ \partial t} \Big) u( x, t) = 0,  \hspace{2cm} (x, t) \in M \times [0, T],
\end{equation}
where the symbol $ \Delta $ is the Laplace-Beltrami operator acting on function in space  with respect to metric $g(t)$ in time. We can couple the Ricci flow (\ref{eq11}) to  equation of the form (\ref{eq12}),  either forward, backward, or perturbed with some potential function, see the author's papers \cite{[Ab1],[Ab2]} for more details. In the present we consider a more generalized situation where the heat equation is perturbed, in this case, the Laplacian is replaced with $ \Delta - R(x,t)$, where $R$ is the scalar curvature of the Ricci flow $g(t)$ and we obtain some Harnack and gradient estimates on the logarithm of the positive solutions.  The author also obtained various estimates on positive solutions and fundamental solution in \cite{[Ab3]}. Throughout, we assume, that the manifold is endowed with bounded curvature, we remark that boundedness and nonnegativity of the curvature is preserved as long as Ricci flow exists \cite{[CLN06]}.

Heat equation coupled to the Ricci flow can be associated with some physical interpretation in terms of heat conduction process. Precisely, the manifold $M$ with initial metric $g(x, 0)$ can be thought of as having the temperature distribution $u(x, 0)$ at $ t = 0.$ If we now allow the manifold to evolve under the Ricci flow and simultaneously allow the heat to diffuse on $M$, then, the solution $u(x, t)$ will represent the space-time temperature on $M$. Moreover, if $u(x, t)$ approaches $\delta$-function at the initial time, we know that $u(x,t) >0$, this gives another physical interpretation that temperature is always positive, whence we can consider the potential $ f = \log u$ as an entropy or unit mass of heat supplied and the local production entropy is given by $ | \nabla f|^2 = \frac{ |\nabla u|^2}{u^2} $.

Harnack inequalities are indeed very powerful tools in geometric analysis. The paper of Li and Yau \cite{[LY86]} paved way for the rigorous studies and many interesting applications of Harnack inequalities. They derived gradient estimates for positive solutions to the heat operator defined on  closed manifold with bounded Ricci curvature and from where their Harnack inequalities follows. These inequalities were in turn used to establish various lower and upper bounds on the heat kernel. They also studied manifold satisfying Dirichlet and Neumann conditions.  On the other hand, Perelman in \cite{[Pe02]} obtained differential Harnack   estimate for the fundamental solution to the conjugate heat equation on compact manifold evolving by the  Ricci flow. Perelman's results  are unprecedented as they play a key factor in the proof of Poincar\'e conjecture. Meanwhile, shortly before Perelman's paper appeared online, C. Guenther \cite{[Gu02]} had found gradient estimates for positive solutions to the heat equation under the Ricci flow by adapting the methodology of Bakry and Qian \cite{[BQ99]} to time dependent metric case. As an application of her results, she got a Harnack-type inequality and obtain a lower bound for fundamental solutions. These techniques were first brought into the study of Ricci flow by R. Hamilton, see \cite{[Ha93a]} for instance. 
As useful  as Harnack inequalities are, they have also been discovered in other geometric flows;  See the following- H-D. Cao \cite{[CaNi]} for heat equation on K\"ahler manifolds, B. Chow \cite{[Ch91b]} for Gaussian curvature flow and \cite{[Ch92]} for Yamabe flow,  also B. Chow and R. Hamilton \cite{[CH97]}, and R. Hamilton \cite{[Ha95b]} on mean curvature flow. 
The following references among many others are found relevant \cite{[Ab4],[Cao08],[Gu02],[Zh06]}. See also the following monographs \cite{[SY94]} on Gradient estimates and \cite{[CCG$^+$07],[CK04],[CLN06],[Mu06]} for theory and application of  Ricci flow.

Recently, \cite{[BCP]} and \cite{[KZh08]} have extended results in \cite{[Zh06]} to heat equation and its conjugate respectively. We remark that our results are similar to those of \cite{[KZh08]} but with different approach, the application to heat conduction that we have in mind has greatly motivated our approach. The detail descriptions of our results are presented in Sections 2 and 3 (localized version), while we collect some elements of the Ricci flow  used in our calculation as an appendix in the last section.
 \section{Estimates on Positive Solutions to the Conjugate Heat Equation}\label{sec2}
Let $\square := \partial_t  - \Delta $ be the heat operator acting on functions $u: M \times [0, T] \rightarrow \mathbb{R}$, where $M \times [0, T]$ is endowed with the volume form $ d \mu(x) dt $. The conjugate to the heat operator $ \Gamma$ is defined by 
\begin{equation}\label{Ceq1}
\square^* = - \partial_t - \Delta_x + R,
\end{equation}
where $R$ is the scalar curvature.
We remark that for any solution $g(t), t \in  [0, T]$ to the Ricci flow and smooth functions $u, v : M \times [0, T] \rightarrow \mathbb{R}$, the following identity holds
\begin{equation}
\int_0^T \int_M ( \square u) v d \mu(x) dt = \int_0^T \int_M u( \square^* v) d \mu(x) dt.
\end{equation}
By direct application of integration by parts with the fact that the functions $u$ and $v$ are $C^2$ with compact support (or if  $M$ is compact) and using evolution of $d \mu$ under the Ricci flow the last identity can be shown easily.
 In a special case $u \equiv 1$, we have 
 $$ \frac{d}{dt} \int_M v d\mu = - \int_M \square^* d\mu. $$
 
 \begin{proposition}\label{prop21}
 Let $u = (4 \pi \tau )^{-\frac{n}{2}} e^{-f}$ be a positive solution to the conjugate heat equation. The evolution equation
 \begin{equation}
 \frac{\partial f}{ \partial t} = - \Delta f + | \nabla f|^2 - R + \frac{n}{ 2 \tau} 
 \end{equation}
 is equivalent to the following evolution 
 \begin{equation}
 \square^* u = 0. 
 \end{equation}
 \end{proposition}
 
 \begin{proof}
 $$  \square^* u = (- \partial_t - \Delta_x + R)(4 \pi \tau )^{-\frac{n}{2}} e^{-f}.$$
 By direct calculation, it follows that
 $$ \partial_t [(4 \pi \tau )^{-\frac{n}{2}} e^{-f}] = ( \frac{n}{ 2\tau}- \partial_t f ) (4 \pi \tau )^{-\frac{n}{2}} e^{-f} $$
 $$ \Delta [(4 \pi \tau )^{-\frac{n}{2}} e^{-f}] =( - \Delta f + | \nabla f|^2 )(4 \pi \tau )^{-\frac{n}{2}} e^{-f}.$$
 Then
 $$\square^* u = \Big( -\frac{n}{ 2\tau}+ \partial_t f + \Delta f - | \nabla f|^2 + R \Big)u = 0, $$
 where we have made use of $ \partial_t \tau = -1$ and since $ u >0$, the claimed is then proved.
 \end{proof}

 Let $(M, g(t)), t \in [0,T]$ be a solution of the Ricci flow on a closed manifold. Here $T > 0$ is taken to be the maximum time of existence for the flow. Let $u$ be a positive solution to the conjugate heat equation, then we have the following coupled system.
  \begin{equation}
 \left \{ \begin{array}{l}\label{Heqn1}
 \displaystyle \frac{\partial g_{ij}}{\partial t} = - 2 R_{ij}  \\
 \ \\
 \displaystyle - \frac{\partial u}{\partial t} - \Delta_{g(t)} u + R_{g(t)} u = 0,
 \end{array} \right. 
  \end{equation}  
which we refer to as Perelman's conjugate heat equation coupled to the Ricci flow. We will prove differential Harnack and gradient estimates for all positive solutions of the conjugate heat equation in the above system. 
A differential Harnack estimate of Li-Yau type yields a space-time gradient estimate for a positive solution to a heat-type equation, which when integrated compares the solution at different points in space and time. We will later apply the maximum principle to obtain a localized version of  the estimates.

\subsection{Main Result I. (Differential Harnack Inequality and Gradient Estimates)}
The main result of this subsection is contained in Theorem (\ref{thm Heqn1}) and as an application we arrived at Theorem (\ref{cor35}) which gives the corresponding Li-Yau type gradient estimate  for all positive solution to the conjugate heat equation in the system (\ref{Heqn1}).

\begin{theorem} \label{thm Heqn1}
Let $ u \in C^{2, 1}(M \times [0,T])$ be a positive solution to the conjugate heat equation $ \square^* u = ( - \partial_t - \Delta + R )u = 0$ and the metric $g(t)$ evolve by the Ricci flow in the interval $[0, T)$ on a closed manifold $M$ with nonnegative scalar curvature. Suppose further  that $u = ( 4 \pi \tau )^{-\frac{n}{2}} e^{-f}$, where $ \tau = T-t$, then for all points $(x,t) \in (M \times [0,T])$, we have the Harnack quantity
 \begin{equation}\label{Heqn2}
  P = 2 \Delta f -| \nabla f|^2 + R - \frac{ 2n}{\tau} \leq 0.
  \end{equation}
Then $P$ evolves as 
  \begin{equation}
  \frac{\partial }{\partial t} P = - \Delta P + 2 \langle \nabla f, \nabla P \rangle + 2 \Big| R_{ij} + \nabla_i \nabla_j f - \frac{1}{\tau} g_{ij} \Big|^2 + \frac{2}{\tau} P + \frac{2}{\tau} | \nabla f|^2 +  \frac{4n}{\tau^2}   + \frac{2}{\tau} R. 
  \end{equation}
  for all $ t > 0$. Moreover $P \leq 0$ for all $t \in [0, T].$ 
\end{theorem}  
  
Note that $u = ( 4 \pi \tau )^{-\frac{n}{2}} e^{-f}$ implies $ \ln u = - f - \frac{n}{2} \ln( 4 \pi \tau)$ and we can write (\ref{Heqn2}) as 
 \begin{equation}
 \frac{| \nabla u|^2}{ u^2} - 2 \frac{u_t}{u}- R - \frac{ 2n}{\tau} \leq 0,
 \end{equation}  
  which is similar to the celebrated Li-Yau \cite{[LY86]} gradient estimate for the heat equation on manifold with nonnegative Ricci curvature.
 
 We need the  usual routine computations as in the following;
 \begin{lemma}\label{lem332}
  Let $(g, f)$ solve the system (\ref{Heqn1}) above. Suppose further that $u = ( 4 \pi \tau )^{-\frac{n}{2}} e^{-f}$ with $ \tau = T -t$. Then we have 
  $$ (  \frac{\partial}{\partial t } + \Delta ) \Delta f = 2 R^{ij} \nabla_i \nabla_j f + \Delta | \nabla f|^2 - \Delta R$$
  and 
  $$(  \frac{\partial}{\partial t } + \Delta )| \nabla f|^2 = 4 R_{ij} \nabla_i f \nabla_j f + 2 \langle \nabla f, \nabla | \nabla f|^2 \rangle + 2 | \nabla \nabla f |^2 - 2 \langle \nabla f, |\nabla R|^2 \rangle. $$ 
  \end{lemma}
  \begin{proof}
  By direct calculation and Proposition \ref{prop21}
   \begin{align*}
  \frac{\partial}{\partial t } ( \Delta f ) =   \frac{\partial}{\partial t  } ( g^{ij} \partial_i \partial_j f ) &  
         =  \frac{\partial}{\partial t } (g^{ij} ) \partial_i \partial_j f  + g^{ij} \partial_i \partial_j  \frac{\partial}{\partial t }f \\
   \displaystyle &=  2R^{ij}  \partial_i \partial_j f  + \Delta ( - \Delta f + | \nabla f|^2 - R +  \frac{n}{ 2 \tau } ) \\
    \displaystyle &=  2R^{ij}  \nabla_i \nabla_j f  -  \Delta( \Delta f) + \Delta | \nabla f|^2 - \Delta R
  \end{align*}
  then,
 \begin{align*}  
 \Big( \frac{\partial}{\partial t }  +\Delta \Big)  \Delta f  &= 2R^{ij}  \nabla_i \nabla_j f -  \Delta( \Delta f) + \Delta | \nabla f|^2 - \Delta R + \Delta( \Delta f)  \\
 \displaystyle &=  2R^{ij}  \nabla_i \nabla_j f + \Delta | \nabla f|^2 - \Delta R
  \end{align*}
  Part 1 is proved.
  \begin{align*}
  \frac{\partial}{\partial t } | \nabla f|^2 & =  2  R^{ij} \partial_i f \partial_j f  + 2 g^{ij} \partial_i f \partial_j \frac{\partial}{ \partial t} f \\ 
\displaystyle &  =  2  R^{ij} \partial_i f \partial_j f   + 2 \langle \nabla f, \nabla( - \Delta f + | \nabla f|^2 - R +  \frac{n}{ 2 \tau } )  \rangle  \\
\displaystyle &  =  2  R^{ij} \nabla_i f \nabla_j f  + 2 \langle \nabla f,  \nabla | \nabla f|^2 \rangle -  2 \langle \nabla f, \nabla  \Delta f  \rangle - 2 \langle \nabla f, \nabla  R \rangle
  \end{align*}
  then,
 \begin{align*}  
 \Big( \frac{\partial}{\partial t }  +\Delta \Big)  | \nabla f|^2  =  2  R^{ij} \partial_i f \partial_j f + 2 \langle \nabla f,  \nabla | \nabla f|^2 \rangle -  2 \langle \nabla f, \nabla  \Delta f  \rangle - 2 \langle \nabla f, \nabla  R \rangle + \Delta | \nabla f|^2.
 \end{align*}
 Using the Bochner identity
 $$ \Delta | \nabla f|^2 = 2 | \nabla \nabla f|^2 + 2 \langle \nabla f, \nabla  \Delta f  \rangle + 2 Rc ( \nabla f, \nabla f) $$ 
 we obtain the identity in part (2).
  \end{proof}
 
 \begin{proof} Proof of Theorem \ref{thm Heqn1}.
 Since $ P = 2 \Delta f - | \nabla f |^2 + R - \frac{2n}{\tau}$ and by direct computation and using Lemma \ref{lem332}, we have 
 
   \begin{align*}
 \Big( \frac{\partial}{\partial t }  +\Delta \Big) P &= 2  \Big( \frac{\partial}{\partial t }  +\Delta \Big) \Delta f -  \Big( \frac{\partial}{\partial t }  +\Delta \Big) |\nabla f |^2 +  \Big( \frac{\partial}{\partial t }  +\Delta \Big) R -   \frac{\partial}{\partial t } \Big(\frac{2n}{\tau} \Big) \\ 
\displaystyle &  =  4  R^{ij} \nabla_i  \nabla_j f  + 2  \Delta | \nabla f|^2 - 2\Delta R - 4 Rc( \nabla f, \nabla f) - 2 \langle \nabla f,  \nabla | \nabla f|^2 \rangle  \\
\displaystyle & \hspace{1.5cm} - 2 | \nabla \nabla f|^2 + 2 \langle \nabla f, \nabla R \rangle + 2 \Delta R + 2 | Rc|^2 + \frac{2n}{\tau^2} \\ 
\displaystyle &  =  4  R^{ij} \nabla_i  \nabla_j f  + 2 | Rc|^2  + \frac{2n}{\tau^2} - 2 \langle \nabla f,  \nabla | \nabla f|^2 \rangle + 2 \langle \nabla f, \nabla R \rangle \\
\displaystyle & \hspace{1.5cm}  + 2 \Delta  | \nabla f|^2 - 4 Rc ( \nabla f,  \nabla f) - 2|  \nabla  \nabla f|^2 \\ 
\displaystyle &  =  4  R^{ij} \nabla_i  \nabla_j f  + 2 | Rc|^2  + \frac{2n}{\tau^2} - 2 \langle \nabla f,  \nabla | \nabla f|^2 \rangle + 2 \langle \nabla f, \nabla R \rangle \\
\displaystyle & \hspace{1.5cm} + \Delta  | \nabla f|^2 - 2 Rc ( \nabla f,  \nabla f) + 2 \langle \nabla f,  \nabla  \Delta f  \rangle \\ 
\displaystyle &  =  4  R^{ij} \nabla_i  \nabla_j f  + 2 | Rc|^2  + \frac{2n}{\tau^2} + 2|  \nabla  \nabla f|^2 - 2 \langle \nabla f,  \nabla | \nabla f|^2 \rangle \\
\displaystyle & \hspace{1.5cm}+ 2 \langle \nabla f, \nabla R \rangle  + 4\langle \nabla f,  \nabla  \Delta f  \rangle \\ 
\displaystyle &  =  4  R^{ij} \nabla_i  \nabla_j f  + 2 | Rc|^2  + \frac{2n}{\tau^2} + 2|  \nabla  \nabla f|^2 + 2 \langle \nabla f,  \nabla P \rangle\\ 
\displaystyle & = 2 | R_{ij} + \nabla_i \nabla_j f |^2 +\frac{2n}{\tau^2}  + 2 \langle \nabla f,  \nabla P \rangle. 
\end{align*}
By direct computation we notice that 
\begin{align*}
 \Big| R_{ij} + \nabla_i \nabla_j f - \frac{1}{\tau} g_{ij} \Big|^2 =  | R_{ij} + \nabla_i \nabla_j f |^2 - \frac{2}{\tau}( R + \Delta f ) + \frac{n}{\tau^2},
\end{align*}
which implies
\begin{align*}
2 | R_{ij} + \nabla_i \nabla_j f |^2 + \frac{2n}{\tau^2} =  2 \Big| R_{ij} + \nabla_i \nabla_j f - \frac{1}{\tau} g_{ij} \Big|^2 + \frac{4}{\tau}( R + \Delta f ).
\end{align*}
Also
\begin{align*}
 \displaystyle \frac{4}{\tau}( R + \Delta f ) &= \frac{2}{\tau}( R + 2\Delta f ) + \frac{2}{\tau} R \\
 \displaystyle & = \frac{2}{\tau} P + \frac{2}{\tau} | \nabla f|^2 +  \frac{4n}{\tau^2}   + \frac{2}{\tau} R.
\end{align*}
Therefore, by putting these together we have 
\begin{align*}
 \Big( \frac{\partial}{\partial t }  +\Delta \Big) P =  2 \langle \nabla f,  \nabla P \rangle + 2 \Big| R_{ij} + \nabla_i \nabla_j f - \frac{1}{\tau} g_{ij}  \Big|^2 +\frac{2}{\tau} P + \frac{2}{\tau} | \nabla f|^2 +  \frac{4n}{\tau^2}   + \frac{2}{\tau} R,
\end{align*}
which proves the evolution equation for $P$.

 To prove that $ P \leq 0$ for all time $t \in [ 0, T]$, we know that for small $ \tau$, $ P(\tau) <0$. We can use the Maximum principle to conclude this. Notice that by the Perelman's $\mathcal{W}$-entropy monotonicity 
 $$ R_{ij} + \nabla_i \nabla_j f - \frac{1}{\tau} g_{ij} \geq 0$$
  and strictly positive except when $g(t)$ is a  shrinking   gradient soliton. So our conclusion will follow from a theorem in \cite[Theorem 4]{[CxZz]}. 
  
  For completeness we show this;
  by Cauchy-Schwarz inequality and the fact that $ R = g^{ij} R_{ij}$ and $ \sum_{i,j} g_{ij} = n$, we have 
 $$ | R_{ij} + \nabla_i \nabla_j f - \frac{1}{\tau} g_{ij} |^2 \geq \frac{1}{n} ( R + \Delta f - \frac{n}{\tau} )^2 $$
 and by definition of $P$
 $$ P + R + | \nabla f |^2 = 2 ( R + \Delta f - \frac{n}{\tau} ). $$

Hence
$$2 \Big| R_{ij} + \nabla_i \nabla_j f - \frac{1}{\tau} \Big|^2 \geq  \frac{1}{2n} (P + R + | \nabla f |^2)^2.$$
Putting the last identity into the evolution equation for $P$ yields
 \begin{align*}
 \frac{\partial P}{\partial t }  & \geq -  \Delta P + 2 \langle \nabla P, \nabla f \rangle + \frac{1}{2n} (P + R + | \nabla f |^2)^2 + \frac{2}{\tau } (P + R + | \nabla f |^2) + \frac{4 n}{\tau^2}\\
 \displaystyle &  = - \Delta P+ 2 \langle \nabla P, \nabla f \rangle +  \frac{1}{2n} (P + R + | \nabla f |^2 + \frac{2n }{\tau^2})^2  +  \frac{2n }{\tau^2 }.
  \end{align*}
This implies that 
 \begin{align*}
 \frac{\partial P}{\partial \tau  } \leq \Delta P- 2 \langle \nabla P, \nabla f \rangle - \frac{1}{2n} (P + R + | \nabla f |^2 + \frac{2n }{\tau^2})^2  -  \frac{2n }{\tau^2 }.
  \end{align*}
Then 
 \begin{align}\label{evq}
 \frac{\partial P}{\partial \tau  } \leq \Delta P- 2 \langle \nabla P, \nabla f \rangle.
  \end{align}
Applying the maximum principle to the evolution equation (\ref{evq}) yields clearly that $P \leq 0$ for all $\tau$, hence, for all $t \in [0, T).$
 \end{proof}
 The result here is an improvement on Kuang and Zhang's \cite{[KZh08]} since it holds with no assumption on the curvature.
 This result can also be compared with those of \cite{[CaH09],[Cao08]} where they define a general Harnack quantity for conjugate heat equation and derive its evolution under the Ricci flow.

\subsection{Main Result II. (Pointwise Harnack Estimates)}
The aim of this subsection is to state and proof the  Li-Yau type pointwise Harnack estimate corresponding to the Harnack inequality proved in the last subsection. We introduce some notations. Given $x_1 , x_2 \in M$ and $t_1, t_2 \in [0,T]$ satisfying $ t_1 < t_2$
$$ \Theta( x_1, t_1; x_2, t_2) = \inf_{\gamma} \int_{t_1}^{t_2} \Big| \frac{d}{dt} \gamma(t) \Big|^2 dt,$$
where the infimum is taken over all the smooth path $ \gamma :[t_1, x_2] \rightarrow M$ connecting $x_1$ and $x_2$. The norm $|.|$ depends on $t$. We now present a lemma which is crucial to the proof of our main result in this subsection.

\begin{lemma}
Let $( M, g(t))$ be a complete solution to the Ricci flow. Let $u: M \times [0,T] \rightarrow \mathbb{R}$ be a smooth positive solution to the heat equation (\ref{eq11}). Define $ f = \log u$ and assumed that 
$$ - \frac{ \partial f}{\partial t} \leq \frac{1}{ \alpha} \Big(  \frac{\beta}{ t} - | \nabla f|^2 \Big), \ \ \ ( x,t) \in M \times [0, T] $$ 
for some $ \alpha, \beta > 0$. Then, the inequality 
\begin{equation}\label{CLem}
u(x_2, t_2 ) \leq u( x_1, t_1) \Big ( \frac{t_2}{t_1}\Big)^{\frac{\alpha}{\beta}} \exp \Big( \frac{\alpha}{4} \Theta( x_1, t_1; x_2, t_2) \Big)
\end{equation}
holds for all $( x_1, t_1)$ and $(x_2, t_2)$ such that $ t_1 < t_2.$
\end{lemma}
\begin{proof}
Obtain the time differential of a function $f$ depending on the  path $\gamma$ as follows
\begin{align*}
\frac{d}{d t} f( \gamma(t), t) &= \nabla  f( \gamma(t), t) \frac{d}{d t} \gamma(t) - \frac{\partial }{\partial s}( \gamma(t), s) \Big|_{ s= t} \\
& \leq  \Big| \nabla f  \Big|  \Big|  \frac{d}{d t} \gamma(t) \Big| +  \frac{1}{ \alpha} \Big(  \frac{\beta}{ t} -  | \nabla f |^2 \Big) \\
& \leq  \frac{\alpha}{4} \Big|  \frac{d}{d t} \gamma(t)\Big|^2 +  \frac{\beta}{ \alpha t}.
\end{align*}
The last inequality was obtained by the application of completing the square method in form of a quadratic inequality satisfying $ax^2 - b x \geq - \frac{b^2}{ 4a} , \ \ ( a , b > 0 )$. Then integrating over the path from $ t_1$ to $ t_2$, we have 
\begin{align*}
 f ( x_2, t_2) - f( x_1, t_1) &= \int_{t_1}^{t_2} \frac{d}{d t} f( \gamma(t), t) dt \\ 
& \leq  \frac{\alpha}{4} \int_{t_1}^{t_2}  \Big|  \frac{d}{d t} \gamma(t)\Big|^2 dt +  \frac{\beta}{ \alpha } \log t \Big|_{t_1}^{t^2}.
\end{align*}
The required estimate ( \ref{CLem}) follows immediately after exponentiation.
\end{proof}

 We have the following as an immediate consequence of the above theorem
 \begin{corollary}\label{cor35} (Harnack Estimates).
 Let $ u \in C^{2,1} ( M \times[0,T))$ be a positive solution to the conjugate heat equation $ \Gamma^* u = 0 $ and $g(t), t \in [0,T)$ evolve by the Ricci flow on a closed manifold $M$ with nonnegative scalar curvature $R$. Then for any points $ (x_1, t_1)$  and $ (x_2, t_2)$ in  $ M \times (0,T)$ such that $ 0 < t_1 \leq t_2 < T$, the following estimate holds
 \begin{equation}
 \frac{u(x_2, t_2)}{u(x_1, t_1)} \leq \Big(  \frac{\tau_1}{\tau_2}\Big)^n \exp \Big[ \int_0^1 \frac{ | \gamma'(s) |^2 }{2 (\tau_1 - \tau_2)} ds + \frac{ (\tau_1 - \tau_2)}{2} R \Big],
 \end{equation}
  where $\tau_i =T - t_i, i =1,2$ and $\gamma : [0,1]$ is a geodesic  curve connecting points $x_1$ and $x_2$ in $M.$
 \end{corollary}
 \begin{proof}
 Let $\gamma : [0,1]$ be a minimizing geodesic connecting points $x_1$ and $x_2$ in $M$ such that $ \gamma(0) = x_1$ and  $ \gamma(1) = x_2$ with $ | \gamma'(s)|$ being the length of the vector  $\gamma'(s)$ at time $ \tau(s) = (1-s) \tau_1 + s \tau_2,  \ \ \ 0 \leq \tau_2 \leq \tau_1 \leq T.$
 Define $\eta(s) = \ln  u( \gamma(s),  (1-s) \tau_1 + s \tau_2)$. Clearly, $\eta(0) = \ln u(x_1, t_1)$ and $\eta(1) = \ln u(x_2, t_2).$
 
 Integrating along $ \eta(s)$, we obtain 
 $$ \ln u(x_2, t_2) - \ln u(x_1, t_1) = \int_0^1 \Big( \frac{ \partial }{\partial s} \ln u(\gamma(s), (1-s) \tau_1 + s \tau_2) \Big) ds $$ 
 i.e.,
 $$ \ln \Big(  \frac{u(x_2, t_2)}{u(x_1, t_1)} \Big) = \ \ln u( \gamma(t), t) \Big|_0^1 .$$
 By direct computation, we have on the path $  \gamma(s)$ that 
  \begin{align*}
  \frac{ \partial }{\partial s} \eta(s) = \frac{ d}{d s} \ln u &= \nabla \ln u \cdot \gamma'(s) +\frac{\partial }{\partial t} \ln u \\
  \displaystyle & =  \frac{\nabla u}{u}  \cdot \gamma'(s) - \frac{u_t (\tau_1 - \tau_2)}{u} \\
    \displaystyle & = (\tau_1 - \tau_2) \Big(   \frac{\nabla u}{u}  \cdot \frac{\gamma'(s)}{\tau_1 - \tau_2}   -  \frac{u_t }{u}  \Big).
   \end{align*}
   From Theorem \ref{thm Heqn1}, we have 
   $$ \frac{| \nabla u |^2}{ u^2} - 2 \frac{u_t}{u} \leq R + \frac{2n}{\tau}, $$ 
   which implies
   $$ - \frac{u_t}{u} \leq \frac{1}{2} ( R + \frac{2n}{ \tau} ) - \frac{| \nabla u |^2}{ 2 u^2}.$$
   By this, we have 
  \begin{align*} 
  \frac{d}{d s} \ln u &\leq (\tau_1 - \tau_2) \Big( \frac{\nabla u}{u} \cdot \frac{\gamma'(s)}{(\tau_1 -\tau_2)} - \frac{|\nabla u |^2}{2 u^2} + \frac{1}{2}( R + \frac{2 n}{ \tau}) \Big)  \\  
   \displaystyle &  = - \frac{(\tau_1 -\tau_2)}{2}\Big( \frac{\nabla u}{u} - \frac{\gamma'(s)}{(\tau_1 -\tau_2)} \Big)^2 \\
   & \hspace{1.5cm} + \frac{(\tau_1 -\tau_2)}{2} \frac{|\gamma'(s)|^2}{(\tau_1 -\tau_2)^2} + \frac{(\tau_1 -\tau_2)}{2}\Big( R + \frac{2 n}{ \tau}\Big)\\ 
    \displaystyle & \leq \frac{|\gamma'(s)|^2}{2 (\tau_1 -\tau_2)}
    + \frac{(\tau_1 -\tau_2)}{2} \Big( R + \frac{2 n}{ \tau}\Big). 
      \end{align*}
 Now integrating with respect to $s$, from $0$ to $1$, we have 
 \begin{equation}
 \ln u \Big|_0^1 \leq \int_0^1 \frac{|\gamma'(s)|^2}{2 (\tau_1 -\tau_2)} + \frac{(\tau_1 -\tau_2)}{2} \int_0^1 R ds + \ln \Big( \frac{\tau_1}{\tau_2} \Big)^n,
 \end{equation}
exponentiating both sides, we get 
  \begin{equation*}
 \frac{u(x_2, t_2)}{u(x_1, t_1)} \leq \Big(  \frac{\tau_1}{\tau_2}\Big)^n \exp \Big[ \int_0^1 \frac{ | \gamma'(s) |^2 }{2 (\tau_1 - \tau_2)} ds + \frac{ (\tau_1 - \tau_2)}{2} R \Big].
 \end{equation*}
 \end{proof}

 \section{Main Result III. (Localising the Harnack and Gradient Estimates)}\label{sec3}
 We establish a localised form of the Harnack and gradient estimates obtained in the last subsection. The main idea is the application of the Maximum principle  on some smooth cut-off function. It was also the basic idea used by Li and Yau in \cite{[LY86]}, this type of approach  has since become tradition. It has been systematically developed over the years since the paper of Cheng and Yau \cite{[CY75]}, see also \cite{[SY94],[Yau75]},  however our computation is more involved as the metric is also evolving.
 
 A natural function that will be defined on $M$ is the distance function from a given point, namely, let $ p \in M$ and define $d(x, p)$ for all $ x\in M,$ where $dist(\cdot, \cdot )$ is the geodesic distance. Note that $d(x, p)$ is only Lipschitz continuous, i.e., everywhere continuously differentiable except on the cut locus of $p$ and on the point where $x$ and $p$ coincide. It is then easy to see that 
$$ | \nabla d | = g^{ij} \partial_i d \ \partial_j d = 1 \ \ on \ \ \ M  \setminus \{ \{ p \}  \cup cut(p) \} .$$
Let $d(x, y, t)$ be the geodesic distance between $x$ and $y$ with respect to the metric $g(t)$,  we  define a smooth cut-off function $ \varphi(x, t)$  with support in the geodesic cube
$$ \mathcal{Q}_{ 2\rho, T} := \{ ( x, t) \in M \times (0, T] : d(x, p, t) \leq 2\rho \} ,$$
for any $C^2$-function $ \psi( s)$ on $[0, + \infty )$ with
\begin{equation*}
\psi(s) = 
\left \{ \begin{array}{l}
 \displaystyle 1 , \hspace{1.5cm} s \in [0,1],  \\
 \displaystyle 0,   \hspace{1.5cm} s \in [2,+ \infty) 
\end{array} \right.  
\end{equation*}
and 
$$ \psi'(s) \leq 0, \ \ \ \frac{| \psi'|^2}{\psi} \leq C_1   \ \ \ and \ \   |  \psi''(s)| \leq  C_2,$$
where $C_1, C_2$ are absolute constants depending only on the dimension of the manifold, such that 
$$ \varphi(x, t) = \psi \Big( \frac{d( x, p, t)}{\rho } \Big)  \ \ \ \ and  \ \  \ \varphi \Big|_{ \mathcal{Q}_{ 2\rho, T} } =1 .$$
We will apply the maximum principle and invoke Calabi's trick \cite{[Ca58]} to assume everywhere smoothness of $ \varphi(x, t)$ since $ \psi(s)$ is in general Lipschitz.
We need Laplacian comparison theorem to do some calculation on $\varphi(x, t)$. Here is the statement of the theorem; Let $M$ be a complete $n$-dimensional Riemannian manifold whose Ricci curvature is  bounded from below by $ Rc \geq (n-1)k $ for some constant $ k \in \mathbb{R}$. Then the Laplacian  of the distance function satisfies 
\begin{equation}\label{eq338}
\Delta d(x, p) \leq  
\left \{ \begin{array}{l}
 \displaystyle  (n-1) \sqrt{k} \cot ( \sqrt{k} \rho) , \hspace{1.8cm} \ k >0 \\ \ \\
 \displaystyle  (n-1) \rho^{-1},    \hspace{3.4cm}  k =0   \\ \ \\
 \displaystyle (n-1) \sqrt{|k|} \coth ( \sqrt{|k|} \rho) , \hspace{1.1cm} \   k < 0.
\end{array} \right.  
\end{equation}
For detail of the Laplacian comparison theorem see \cite[Theorem 1.128]{[CLN06]} or the book \cite{[SY94]}. We are now set to prove the localized version of the gradient estimate for the system (\ref{Heqn1}).
\begin{theorem} \label{thmL}
Let $ u \in C^{2, 1}(M \times [0,T])$ be a positive solution to the conjugate heat equation $ \Gamma^* u = ( - \partial_t - \Delta + R )u = 0$ defined in geodesic cube $ \mathcal{Q}_{ 2\rho, T}$ and the metric $g(t)$ evolves by the Ricci flow in the interval $[0, T]$ on a closed manifold $M$ with bounded Ricci curvature, say $ Rc \geq -Kg$, for some constant $K > 0$. Suppose  further that $u = ( 4 \pi \tau )^{-\frac{n}{2}} e^{-f}$, where $ \tau = T-t$, then for all points in $ \mathcal{Q}_{ 2\rho, T}$ we have the following estimate
 \begin{equation}\label{eq410}
 \frac{| \nabla u|^2}{ u^2} - 2 \frac{u_t}{u}- R  \leq \frac{4n}{ 1 - 4  \delta n} \Bigg \{ \frac{1}{\tau}  + C \Bigg( \frac{1}{\rho^2} + \frac{\sqrt{K}}{\rho} + \frac{K}{\rho} + \frac{1}{T}  \Bigg) \Bigg \},
  \end{equation}
 where $C$ is an absolute constant depending only on the dimension of the manifold and $\delta$ such that $\delta < \frac{1}{4n}.$
\end{theorem}
\begin{proof} 
Recall the evolution equation for the differential Harnack quantity  $$  P = 2 \Delta f -| \nabla f|^2 + R - \frac{ 2n}{\tau},  $$
   \begin{equation*}
  \frac{\partial }{\partial t} P \geq - \Delta P + 2 \langle \nabla f, \nabla P \rangle + 2 | R_{ij} + \nabla_i \nabla_j f - \frac{1}{\tau} g_{ij}|^2 + \frac{2}{\tau} P +  \frac{ 4n}{\tau^2} +  \frac{ 2}{\tau} |\nabla f|^2,
  \end{equation*}
  using the non negativity of the scalar curvature
  Multiplying the quantity $P$ by $t \varphi$, since $ \varphi$ 
is time-dependent we have at any
 point where $ \varphi \neq 0$ that
\begin{align*}
 \frac{1}{\tau } \frac{\partial }{\partial t}( \tau  \varphi P)& = \varphi  \frac{\partial P}{\partial t} +  \frac{\partial \varphi }{\partial t} P -\frac{ \varphi P}{\tau}\\
& \geq  \varphi \Big( - \Delta P + 2 \langle \nabla f, \nabla P \rangle + \frac{2}{\tau} P + \frac{ 4n}{\tau^2} +  \frac{ 2}{\tau} |\nabla f|^2 \Big) \\
& + 2 \varphi | R_{ij} + \nabla_i \nabla_j f - \frac{1}{\tau} g_{ij}|^2  +  \frac{\partial \varphi }{\partial t} P - \frac{ \varphi P}{\tau}\\ 
& =  - \Delta ( \varphi P) + 2 \nabla \varphi \nabla P + 2 \langle \nabla f, \nabla P \rangle \varphi + P( \Delta + \partial_t ) \varphi \\
& +   \frac{ 4n}{\tau^2} \varphi + \frac{ \varphi P}{\tau} + \frac{ 2}{\tau} \varphi  |\nabla f|^2 + 2 \varphi | R_{ij} + \nabla_i \nabla_j f - \frac{1}{\tau} g_{ij}|^2.  
\end{align*}
The last equality is due to derivative test on  $(\varphi P)$ at the minimum point as obtained in the condition (\ref{mincond}) below.
The approach is to estimate $ \frac{\partial}{\partial t}( \tau  \varphi P)$ at the point where minimum (or maximum) value for $(\tau  \varphi P)$ is attained and do some analysis at the minimum (or maximum) point. We know that the support of $(\tau  \varphi P)(x, t)$ is contained in $ \mathcal{Q}_{2\rho} \times [0, T]$ since 
$$ Supp ( \varphi)  \subset \mathcal{Q}_{ 2\rho, T} := \{ ( x, t) \in M \times (0, T] : d(x, p, t) \leq 2\rho \} .$$
Now  let $ (x_0, t_0)$ be a point in $\mathcal{Q}_{ 2\rho, T}$ at which $( \tau  \varphi P)$ attains its minimum value. At this point, we have to assume that $P$ is positive since if $P \leq 0$, we have the same estimate and $( \tau  \varphi P)(x_0, t_0) \leq 0$ implies $( \tau  \varphi P)(x, t) \leq 0$ for all $ x \in M$ such that the distance $d( x, x_0, t) \leq 2\rho $ and the theorem will follow trivially.

Note that at the minimum point $( x_0, t_0)$ we have by the derivative test that  $( 0 \leq \varphi \leq 1)$
 \begin{align}\label{eq215}
   \nabla ( \tau  \varphi P) (x_0, t_0) = 0 , \  \ \ \frac{\partial }{\partial t} ( \tau  \varphi P) (x_0, t_0) \leq 0 \ \ \ and \ \ \ \Delta ( \tau  \varphi P)(x_0, t_0) \geq 0.
    \end{align}
    We shall obtain a lower bound for $ \tau  \varphi P$ at this minimum point.
Therefore    
  \begin{align}\label{eq400}
 \left. \begin{array}{l}
 \displaystyle 0 \geq - \Delta ( \varphi P) + 2 \nabla \varphi \nabla P + 2 \langle \nabla f, \nabla P \rangle \varphi + P( \Delta + \partial_t ) \varphi +  \frac{  \varphi P}{\tau}  \\ 
 \displaystyle  \hspace{1cm} + \frac{ 4n}{\tau^2} \varphi + \frac{ 2}{\tau} \varphi  |\nabla f|^2 + 2 \varphi | R_{ij} + \nabla_i \nabla_j f - \frac{1}{\tau} g_{ij}|^2.  
\end{array} \right.
 \end{align} 
By the argument in (\ref{eq215}) and product rule we have
   $$ \nabla( \varphi P) (x_0, t_0) - P \nabla \varphi (x_0, t_0) = \varphi \nabla P (x_0, t_0) $$
   which means $ \varphi \nabla P $ can always be replaced by $ - P
\nabla \varphi$. Similarly, 
 \begin{align}\label{mincond}
   - \varphi \Delta P =   - \Delta ( \varphi P)  + P \Delta  \varphi + 2 \nabla \varphi \nabla P,
   \end{align} 
   which we have already used before the last inequality. 
   Notice that by direct calculation using product rule
$$  \nabla \varphi \nabla P =  \frac{ \nabla \varphi }{\varphi} \cdot \nabla( \varphi P ) - \frac{| \nabla \varphi|^2}{\varphi} P$$
and
$$ 2 \langle \nabla f, \nabla P \rangle \varphi  =  \langle \nabla f, \nabla( \varphi P) \rangle -  \langle \nabla f, \nabla \varphi \rangle   P. $$
Putting the last two equations into (\ref{eq400}) we have 
\begin{align*}
0&\geq  - \Delta ( \varphi P)   + 2  \frac{ \nabla \varphi }{\varphi} \cdot \nabla( \varphi P ) - 2  \frac{| \nabla \varphi|^2}{\varphi} P + 2 \langle \nabla f, \nabla( \varphi P) \rangle  -2  \langle \nabla f, \nabla \varphi \rangle   P  \\
& + P( \Delta + \partial_t ) \varphi   +  \frac{  \varphi P}{\tau} + \frac{ 4n}{\tau^2} \varphi + \frac{ 2}{\tau} \varphi  |\nabla f|^2  + 2 \varphi | R_{ij} + \nabla_i \nabla_j f - \frac{1}{\tau} g_{ij}|^2 .
\end{align*}
By using the argument in (\ref{eq215})
 \begin{align}\label{eq413}
 \left. \begin{array}{l}
 \displaystyle 0 \geq    - 2  \frac{| \nabla \varphi|^2}{\varphi} P    -2  \langle \nabla f, \nabla \varphi \rangle   P  + P( \Delta + \partial_t ) \varphi   +  \frac{\varphi P}{\tau}\\      
 \displaystyle      \hspace{1cm} + \frac{ 4n}{\tau^2} \varphi + \frac{ 2}{\tau} \varphi  |\nabla f|^2  + 2 \varphi | R_{ij} + \nabla_i \nabla_j f - \frac{1}{\tau} g_{ij}|^2 
\end{array} \right \}.
 \end{align} 
 Observe that for any $ \delta >0$,  
 $$  2  | \nabla f| | \nabla \varphi| P = 2 \varphi| \nabla f| \frac{| \nabla \varphi |}{\varphi} P  \leq \delta \varphi | \nabla f|^2 P + \delta^{-1} \frac{| \nabla \varphi|^2}{\varphi} P $$
 \begin{equation}\label{eq414}
   2  | \nabla f| | \nabla \varphi| P \leq \delta \varphi | \nabla f|^4 P  + \delta \varphi P^2 + \delta^{-1} \frac{| \nabla \varphi|^2}{\varphi} P 
  \end{equation}
  and also that 
  $$ | R_{ij} + \nabla_i \nabla_j f - \frac{1}{\tau} g_{ij}|^2  \geq \frac{1}{n} \Big( R + \Delta f - \frac{n}{\tau} \Big)^2. $$ 
  It is equally clear that 
  $$ P = 2 \Delta f - | \nabla f |^2 + R - \frac{ 2n}{\tau} = 2  \Big( R + \Delta f - \frac{n}{\tau} \Big)^2 - | \nabla f |^2 - R ,$$
  which implies
  $$( P + | \nabla f |^2 + R ) = 2  \Big( \Delta f + R - \frac{n}{\tau} \Big).$$
 Therefore
 $$ 2 \varphi | R_{ij} + \nabla_i \nabla_j f - \frac{1}{\tau} g_{ij}|^2  \geq \frac{ \varphi}{2 n} \Big(  P + | \nabla f |^2 + R \Big)^2. $$  
Notice also that 
\begin{align*}
( P + | \nabla f |^2 + R )^2 (y, s) & = (P + | \nabla f |^2 + R^+ - R^- )^2(y, s) \\
& \geq \frac{1}{2}(  P + | \nabla f |^2 + R^+)^2(y,s) - (R^-)^2(y,s) \\
& \geq \frac{1}{2}(  P + | \nabla f |^2 )^2(y,s) - (R^-)^2(y,s) \\
&  \geq \frac{1}{2}(  P^2 + | \nabla f |^4 )(y,s) - ( \sup_{\mathcal{Q}_{2 \rho, T}} R^-)^2 \\
&  \geq \frac{1}{2}(  P^2 + | \nabla f |^4 )(y,s) - n^2 K^2,
\end{align*} 
where we have applied some inequalities, namely; $2(a -b)^2 \geq a^2 - 2 b^2$ and $( a +b)^2 \geq a^2 + b^2 $ with $ a , b \geq 0$ and a lower bound assumption on Ricci curvature, $ R_{ij} \geq - K,$ which implies $R = -nK\ \implies R^- \leq nK$ and $R = -R^-$. 
 Hence
 \begin{equation}\label{eq415}
  2 \varphi | R_{ij} + \nabla_i \nabla_j f - \frac{1}{\tau} g_{ij}|^2  \geq  \frac{\varphi}{4 n}  P^2 + \frac{\varphi}{4 n} | \nabla f|^4 .  
   \end{equation}
  Whereever $ P <  0$, we then obtain from (\ref{eq413}) -  (\ref{eq415}) that 
  \begin{align*}
0& \geq   \Big( \frac{1 }{4 n} - \delta  \Big) \varphi P^2 + \Bigg \{ ( \delta^{-1} -  2)  \frac{| \nabla \varphi|^2}{\varphi}   + ( \Delta + \partial_t ) \varphi  + \frac{\varphi}{\tau}  \Bigg \} P \\ 
& \hspace{3cm} - \Big( \delta - \frac{1 }{4 n} \Big) \varphi | \nabla f|^4 + \frac{2 }{\tau} \varphi | \nabla f|^2 + \frac{ 4 n}{\tau^2} \varphi,
\end{align*}
using the inequality of the form $ m | \nabla f|^4 - n | \nabla f|^2 \geq -  \frac{n^2}{4 m}$ and
multiplying by $ \varphi$ again ($ \varphi \neq 0$), we have  a quadratic polynomial in $ ( \varphi P )$ which we use to bound $( \varphi P)$ in the following 
 \begin{align}\label{eq416}
 \left. \begin{array}{l}
 \displaystyle  \Big( \frac{1 }{4 n} - \delta  \Big) ( \varphi P )^2 +  \Bigg \{ ( \delta^{-1} - 2)  \frac{| \nabla \varphi|^2}{\varphi}   + ( \Delta + \partial_t ) \varphi  + \frac{\varphi}{\tau}  \Bigg \}  ( \varphi P) \\ \ \\
 \displaystyle  \hspace{3cm}  - \frac{ 4 n}{\tau^2} \Big( \frac{1}{ 1 - 4 n \delta } - 1 \Big) \varphi^2 \leq	 0.
\end{array} \right \}.
\end{align}
Note that if there is a number $x \in \mathbb{R}$ satisfying inequality $ px^2 + q x + r \leq 0, $ when $ p > 0, q > 0$ and $r <0$, then $ q^2 - 4pr > 0$ and we then have the bounds
$$ \frac{-q - \sqrt{q^2 - 4pr}}{ 2p} \leq x \leq  \frac{-q + \sqrt{q^2 - 4pr}}{ 2p}, $$
which clearly implies
$$ \frac{-q - \sqrt{ - 4pr}}{ p} \leq x \leq  \frac{q + \sqrt{- 4pr}}{ p}. $$
Now,  choosing $\delta$ such that $ \delta < \frac{1}{4n}$ and denoting 
$$ Z =  ( \delta^{-1} - 2)  \frac{| \nabla \varphi|^2}{\varphi}   + ( \Delta + \partial_t ) \varphi, $$
we obtain
$$ \tau_0 \varphi P \geq - \frac{4n}{ 1 - 4  \delta n} \Bigg \{ \tau_0 Z + \varphi + 4 \varphi\sqrt{ { \delta n}}  \Bigg \}.$$
Moreover, since $\tau_0 \leq \tau \leq T $ and $ 0 \leq \varphi \leq 1$, we have
$$ \tau P \geq - \frac{4n}{ 1 - 4  \delta n}  \Big \{ \tau Z + 1 + C_3 \Big \},$$
where $C_3$ depends on $n$ and $ \delta$.
It remains to estimate $Z$ 
via  appropriate choice of a cut function $ \varphi : M \times [0, T] \rightarrow [0, 1]$
 such that $ \frac{\partial }{\partial t} \varphi, \Delta \varphi$ and 
 $  \frac{| \nabla \varphi|^2}{\varphi} $ have appropriate upper bounds. 
 The main difficulty with this kind of approach lies in the fact that for any 
 cut-off function, one gets different kind of estimates and therefore the 
 cut-off function in use must be chosen so related to the result one is looking for.

Define a $ C^2$-function $ 0 \leq \psi \leq 1$, on $[0, \infty )$ satisfying 
$$ \psi'(s) \leq 0, \ \ \ \frac{| \psi'|^2}{\psi} \leq C_1 \ \ \ and \ \   |  \psi''(s)| \leq  C_2$$
and define $\varphi$ by 
$$ \varphi(x, t) = \psi \Big( \frac{d( x, x_0, t)}{\rho } \Big) $$
and we have the following after some computations
\begin{align*}
 \frac{| \nabla \varphi |^2}{ \varphi} = \frac{| \psi' |^2 \cdot | \nabla d|^2 }{\rho^2 \varphi} \leq \frac{C_2}{\rho^2},
  \end{align*}
and by the Laplacian comparison Theorem (\ref{eq338}) we have
\begin{align*}
 \Delta \varphi &= \frac{ \psi' \Delta d}{\rho} + \frac{ \psi'' | \nabla d |^2}{\rho^2}  \leq \frac{ C_1}{\rho}  \sqrt{K} + \frac{C_2}{\rho^2}
  \end{align*}
  
   Next is to estimate time derivative of $\varphi$: consider a fixed smooth path $\gamma :[a, b] \to M$ whose length at time $t$ is given by $d(\gamma) = \int_a^b |\gamma'(t)|_{g(t)} dr$, where $r$ is the arc length. Differentiating we get 
 $$ \frac{\partial}{\partial t} (d(\gamma)) = \frac{1}{2} \int_a^b \Big|\gamma'(t) \Big|^{-1}_{g(t)} \frac{\partial g}{\partial t} \Big(\gamma'(t), \gamma'(t)\Big) dr = \int_\gamma Rc(\xi, \xi) dr,$$
 where $\xi$ is the unit tangent vector to the path $\gamma$. For detail see \cite[Lemma 3.11]{[CK04]}. Now  
\begin{align*}
\frac{\partial}{\partial t} \varphi = \psi' \Big(\frac{d}{\rho} \Big) \frac{1}{\rho} \frac{d}{ dt} (d(x, p, t)) & = \psi' \Big(\frac{d}{\rho} \Big) \frac{1}{\rho} \int_{\gamma} Rc ( \xi(s), \xi(s) ) ds \\  & \leq \frac{\sqrt{C_1}}{\rho}  \psi^{\frac{1}{2}} K.
\end{align*}
Therefore 
$$ Z  \leq \frac{C_2'}{\rho^2} + \frac{ C_1}{\rho}  \sqrt{K} + \frac{\sqrt{C_1}}{\rho} K + \frac{C_2}{\rho^2},$$
where $C_2'$ depends on $n$ and $ \delta.$ 
Hence
\begin{align*}
\varphi P \geq  - \frac{4n}{ 1 - 4  \delta n}  \Bigg \{ \frac{1}{\tau}  + C \Bigg( \frac{1}{\rho^2} + \frac{\sqrt{K}}{\rho} + \frac{K}{\rho} + \frac{1}{\tau}  \Bigg) \Bigg \},
\end{align*}
where $ C = \max \{C_1, C_2,  C_3 \}.$ The required estimate follows since both minimum and maximum points for $( \varphi P)$  are contained in the cube $\mathcal{Q}_{2\rho, T}.$
\end{proof}

\section{Concluding Remarks}
Our main results in Section \ref{sec2} hold for all positive solutions  and calculations 
are done without recourse to reduced length, therefore, they can be seen 
as  improvement on Perelman's which only works for the fundamental 
solution via his reduced distance function.  After a simple modification and $\epsilon $
 regularisation  method we can get a corresponding result for heat kernel-type function, namely, 
 if the function $ u(x,t)$ is defined on $ M \times (0, T]$ instead of $ M \times [0, T]$, 
 it suffices to replace $u(x, t)$ and $g(x,t)$ with $u(x, t+ \epsilon )$ and $g(x,t + \epsilon )$ 
 for a sufficiently small $\epsilon  >0$, do similar analysis and later send $\epsilon $ to $0$.
The local estimate in
Section \ref{sec3}  is desirable to extend our result to the case the manifold is noncompact,
for example, in the local monotonicity formula and mean value theorem 
considered in \cite{[EKNT]} a local version is needed. The  estimates 
obtained here are used to prove on diagonal and gaussian-type upper bound for 
heat kernel under a mild assumption on curvature and a technical lemma 
involving the best constant in the Sobolev embedding. This will be announced 
in a forthcoming paper. 

\section*{Appendix} 
\subsection*{Elements of the Ricci Flow}  Given an $n$-dimensional Riemannaian manifold $M$ endowed with metric $g$. In local coordinate $ \{ x^i \}, 1 \leq i \leq n $, we can write the metric in component form
$$ ds^2 = g = g_{ij} dx^i d x^j. $$
Consider a smooth function $f$ defined on $M$, then, the Laplace-Beltrami operator acting on $f$ is defined by 
$$ \Delta f = \frac{1}{ \sqrt{ |g|} }\sum_{i,j}^n \frac{\partial}{\partial x^i} \Big( \sqrt{|g|} g^{ij} \frac{\partial}{\partial x^j} f \Big)  = g^{ij} \Big( \partial_i \partial_j f - \Gamma_{ij}^k \partial_k f \Big),$$
where $( g^{ij} ) = ( g_{ij})^{-1} $ is the metric inverse,  $|g|$ is the matrix determinant of $( g^{ij} )$ and $ \Gamma_{ij}^k$ are the Christoffel's symbols.
The degenerate parabolic partial differential equation
$$ \frac{\partial }{ \partial t} g_{ij} = - 2 R_{ij} $$
is the Ricci flow on $(M, g(t))$, where $R_{ij}$ is the component of the Ricci curvature tensor and $g(t)$ is a one-parameter family of Riemannian metrics. The degeneracy of the pde is due to the group of differomorphism invariance, but we are sure of the existence of solution at least for short a time (Hamilton \cite{[Ha82]}). 

Interestingly all geometric quantities associated with $M$ also evolve along the Ricci flow, in the present study we have made use of the following evolutions
\begin{align*}   
 \displaystyle & metric \ inverse :  \ \ \ \ \ \ \ \ \ \ \  \frac{\partial}{\partial t } g^{ij} = 2 R^{ij}\\
 \displaystyle  &  volume \ element : \  \ \ \ \ \ \ \ \ \   \frac{\partial}{\partial t } d \mu \   = - R d \mu \\
  \displaystyle &  Scalar \ curvature : \ \ \ \ \ \ \ \ \frac{\partial}{\partial t} R  =  \Delta R + 2 | R_{ij}|^2 \\
   \displaystyle & Laplacian:  \ \ \ \  \ \ \ \ \ \ \ \ \ \ \ \ \ \ \frac{\partial}{\partial t}  \Delta_{g(t)}  =  2 R_{ij} \cdot \nabla^i \nabla^j \\
    \displaystyle & Christoffel's  \ symbols: \   \frac{\partial}{\partial t } \Gamma^k_{ij} =  - g^{kl} \Big( \partial_i R_{jl} +  \partial_j R_{il} -   \partial_l R_{ij} \Big),
\end{align*}
where $R_{ij}$ is the Ricci curvature ($R = g^{ij} R_{ij}$) and $ \sum_{ij} g^{ij} = n$. 
The metric is bounded under the Ricci flow (Cf. \cite{[CCG$^+$07],[CK04],[CLN06]}, for  more on this and detail of geometric and analytical aspect of the Ricci flow).

\section*{Acknowledgment}
The author  wishes to  acknowledge his PhD thesis advisor Prof. Ali Taheri for constant encouragement. He also thanks the anonymous reviewers for their valuable comments.
His research is supported by TETFund of Federal Government of Nigeria and University of Sussex, United Kingdom.

\end{document}